\documentclass[12 pt,english,oneside]{amsart}
\usepackage[top=1.16in, bottom=1.16in, left=1.23in, right=1.23in]{geometry}
\usepackage[latin1]{inputenc}
\usepackage[T1]{fontenc}
\usepackage[normalem]{ulem}
\usepackage{verbatim} 
\usepackage{graphicx}
\usepackage{xypic} 
\usepackage{amsfonts}
\usepackage{amsmath}
\usepackage{amssymb}
\usepackage{graphicx,epsfig}
\usepackage{amsthm}
\usepackage{amscd}
\usepackage{dsfont}
\usepackage{fancyhdr}

\newcommand{\g}{\mathfrak{g}}

\newcommand{\ra}{\rightarrow}
\newcommand{\mh}{\mathfrak{h}}
\newcommand{\ve}{\varepsilon}
\newcommand{\vp}{\varphi}

\newcommand{\bino}[2]{\left[\genfrac{}{}{0pt}{}{#1}{#2}\right]}
\newcommand{\ts}{\otimes}
\newcommand{\s}{\sigma}

\newcommand{\pp}{\mathcal{P}}
\newcommand{\mz}{\mathbb{Z}}
\newcommand{\mc}{\mathbb{C}}
\newcommand{\fb}{\mathfrak{B}}

\newcommand{\wt}{\widetilde}

\newcommand{\id}{\operatorname*{id}}
\newcommand{\Tr}{\operatorname*{Tr}}
\newcommand{\End}{\operatorname*{End}}
\newcommand{\Aut}{\operatorname*{Aut}}

\numberwithin{equation}{section}

\newtheorem{definition}{Definition}
\newtheorem{proposition}{Proposition}
\newtheorem{theorem}{Theorem}
\newtheorem{lemma}{Lemma}

\newtheorem{remark}{Remark}

\title{Dedekind $\eta$-function and quantum groups}
\author{Xin Fang}
\address{Mathematisches Institut, Universit\"{a}t zu K\"{o}ln, Weyertal 86-90, D-50931 K\"{o}ln.}
\email{xinfang.math@gmail.com}
\keywords{quantum groups, Dedekind $\eta$-function, braid goups}
\subjclass[2010]{17B37, 20G42, 81R50}

\begin{document}
\maketitle

\begin{abstract}
We realize some powers of Dedekind $\eta$-function as traces on quantum coordinate algebras.
\end{abstract}

\section*{Introduction}
\subsection*{History}
The partition function $p(n)$ of a positive integer $n$ and its numerous variants have a long history in combinatorics and number theory. A basic method to study these functions defined on the set of integers is considering their generating functions (for example: $\psi(x)=\sum_{n\geq 0}p(n)x^n$) to study their analytical properties, the algebraic equations they satisfy or the (quasi-)symmetries under group actions and so on.
\par
In the case of partition function, the inverse $\psi(x)^{-1}$ of its generating function admits the following description
$$\vp(x):=\psi(x)^{-1}=\prod_{n\geq 1}(1-x^n).$$
The function $\psi(x)$ is closed related to modular forms: for instance, $\eta(x)=x^{\frac{1}{24}}\vp(x)$ is the Dedekind $\eta$-function and up to a scalar, $\Delta(x)=\eta(x)^{24}$ is a modular form of weight $12$ whose Fourier coefficients
give the famous Ramanujan's $\tau$-function.
\par
Powers of $\vp(x)$ are initially studied by Euler and Jacobi: Euler discovered a relation between $\vp(x)$ and the pentagon numbers
$$\vp(x)=\sum_{n\in\mathbb{Z}}(-1)^n x^{\frac{3n^2-n}{2}}$$
and Jacobi deduced the expression of $\vp(x)^3$ in terms of triangle numbers in the study of elliptic functions:
$$\vp(x)^3=\sum_{n=0}^\infty (-1)^n (2n+1) x^{\frac{n(n+1)}{2}}.$$

\subsection*{Work of MacDonald and Kostant}
These formulae involving powers of the Dedekind $\eta$-function are generalized in the seminal work of I. MacDonald \cite{Mac} by interpreting them as special cases of the Weyl denominator formula of some affine root systems: the Jacobi identity above can be obtained using combinatorial data from the affine root system of type $A_1$.
\par
To be more precise, for any reduced root system on a finite dimensional real vector space $V$ with the standard bilinear form $(\cdot,\cdot)$, we can associate to it a complex Lie algebra $\g$. The following identity (formula (0.5) in \cite{Mac}) is a specialization of the Weyl denominator formula:
\begin{equation}\label{0.1}
\eta(x)^d=\sum_{\mu\in M}d(\mu) x^{(\mu+\rho,\mu+\rho)/2g},
\end{equation}
where $M$ is a certain explicitly described subset of the set of dominant integral weights, $d=\textrm{dim}\g$, $d(\mu)$ is the dimension of the irreducible representation of $\g$ of highest weight $\mu$, $g=\frac{1}{2}((\phi+\rho,\phi+\rho)-(\rho,\rho))$, $\phi$ is the highest root of $\g$ and $\rho$ is half sum of positive roots.
\par
B. Kostant \cite{Kostant} gave another description of $M$ by connecting it with the trace of a Coxeter element $\bold{c}$ in the Weyl group $W$ acting on the subspace of weight zero $V_1(\lambda)_0$ in the irreducible representation $V_1(\lambda)$ associated to a dominant integral weight $\lambda\in\mathcal{P}_+$. If the Lie algebra is simply-laced (i.e., of type A,D,E), Kostant's formula reads:
\begin{equation}\label{0.2}
\eta(x)^{d}=\sum_{\lambda\in\mathcal{P}_+}\Tr (\bold{c},V_1(\lambda)_0) \textrm{dim}V_1(\lambda) x^{(\lambda+\rho,\lambda+\rho)}.
\end{equation}
A similar result valid for a general $\g$ can be found in \cite{Kostant}, Theorem 1.
\par
A good summary of these works can be found in a Bourbaki seminar talk \cite{Demazure} by M. Demazure.

\subsection*{Main results}
The main goal of this paper is to prove identities in the spirit of formulae (\ref{0.1}) and (\ref{0.2}) in the framework of quantum groups, giving compact forms of identities cited above.
\par
Let $\g$ be a finite dimensional simple Lie algebra and $U_q(\g)$ be the associated quantum group over $\mc(q)$. The Artin braid group $\mathfrak{B}_\g$ associated to the Weyl group $W$ of $\g$ acts on the irreducible representation $V(\lambda)$ of $U_q(\g)$ with highest weight $\lambda\in\mathcal{P}_+$. Let $\{\s_1,\cdots,\s_l\}$ be the set of generators of $\mathfrak{B}_\g$, $\Pi=\s_1\cdots\s_l$ be a Coxeter element and $h$ be the Coxeter number of the Weyl group. Then $\Pi\ts \id$ acts on the quantum coordinate algebra $\mc_q[G]=\bigoplus_{\lambda\in\mathcal{P}_+}V(\lambda)\ts V(\lambda)^*$ componentwise (see the argument before Theorem \ref{Thm:main} for details) and we obtain finally
\\
\\
\noindent
\textbf{Theorem.} The following identity holds:
$$\Tr (\Pi\ts\id,\mc_q[G])=\left(\prod_{i=1}^l\vp(q^{(\alpha_i,\alpha_i)})\right)^{h+1}.$$

\subsection*{Organization of this paper}
In Section \ref{s2} and \ref{s3} the basic definitions and properties of Lie algebras and quantum groups are recalled.
Section \ref{s5} is devoted to computing the action of a Coxeter element in the Artin braid group on the irreducible representations of the quantum group $U_q(\g)$, which leads to the main theorem in Section \ref{s6}.

\subsection*{Acknowledgements}
This work was part of the author's Ph.D thesis, supervised by Professor Marc Rosso. I am grateful to his guidance and treasurable discussions on this problem. I would like to thank Jiuzu Hong for pointing out an error in earlier computations. I am grateful to the anonymous referee, whose remarks simplify the paper substantially. The revision of this paper is carried on when the author is supported by the Alexander von Humboldt Foundation.

\section{Quantum groups}\label{s2}

We recall the necessary notations and definitions for Lie algebras and corresponding quantum groups.

\begin{enumerate}
\item $\g$ is a finite dimensional simple Lie algebra over $\mathbb{C}$ with a fixed Cartan subalgebra $\mathfrak{h}$. We let $l=\textrm{dim}\mathfrak{h}$ denote the rank of $\g$ and $I$ be the index set $\{1,\cdots,l\}$.
\item $\Phi:\g\times\g\ra\mc$ is the Killing form. Its restriction to $\mathfrak{h}$ is non-degenerate and then induces a bilinear form on $\mathfrak{h}^*$ which is also denoted by $\Phi$.
\item $\Delta_+\subset\mathfrak{h}^*$ is the set of positive roots of $\g$ and $\{\alpha_1,\cdots,\alpha_l\}$ is the set of simple roots of $\g$.
\item $C=(c_{ij})_{l\times l}$, where $c_{ij}=2\Phi(\alpha_i,\alpha_j)/\Phi(\alpha_i,\alpha_i)$, is the Cartan matrix of $\g$.
\item $W$ is the Weyl group of $\g$ generated by simple reflections $s_i:\mh^*\ra\mh^*$ for $i\in I$ where $s_i(\alpha_j)=\alpha_j-c_{ij}\alpha_i$.
\item A Coxeter element is a product of all simple reflections, any two Coxeter elements are conjugate in $W$.  We let $h$ be the Coxeter number of the Weyl group $W$: it is by definition the order of a Coxeter element.
\item $(\cdot,\cdot):\mh^*\times\mh^*\ra\mathbb{Q}$ is a unique $W$-invariant scalar product such that $(\alpha,\alpha)=2$ for all short roots $\alpha$. We denote $A=(a_{ij})_{l\times l}$ be the matrix with $(\alpha_i,\alpha_j)=a_{ij}$: it is proportional to the Killing form. There exists a constant $k\in\mc^*$ such that for any $x,y\in\mh^*$, $k\Phi(x,y)=(x,y).$ We define $r_\g=k/h\in\mathbb{Q}$. 
\item $D=\textrm{diag}(d_1,\cdots,d_l)$ is the diagonal matrix with $d_i=(\alpha_i,\alpha_i)/2$. Then the matrix $A=DC$. 
\item $\mathcal{Q}=\mz\alpha_1+\cdots+\mz\alpha_l$ is the root lattice and $\mathcal{Q}_+=\mathbb{N}\alpha_1+\cdots+\mathbb{N}\alpha_l$.
\item $\pp$ is the weight lattice and $\pp_+$ is the set of dominant integral weights.
\item For $\lambda\in\pp_+$, $V_1(\lambda)$ is the finite dimensional irreducible representation of $\g$ of highest weight $\lambda$.
\item The coroot $\alpha^\vee$ of a root $\alpha$ is denoted by $\alpha^\vee=2/(\alpha,\alpha)$.
\item $\rho$ is the half sum of all positive roots.. For $\lambda\in\mathcal{P}_+$, we denote $c(\lambda)=\Phi(\lambda+\rho,\lambda+\rho)-\Phi(\rho,\rho)$.
\end{enumerate}

We assume that $q$ is a variable and $q_i=q^{d_i}$. The $q$-numbers and $q$-exponentials are defined by:
$$[n]_q=\frac{q^n-q^{-n}}{q-q^{-1}},\ \ [n]_q!=\prod_{i=1}^n [i]_q,\ \ \textrm{exp}_q(x)=\sum_{k=0}^\infty \frac{1}{[k]_q!}q^{\frac{k(k-1)}{2}}x^k.
$$
\begin{definition}
The quantized enveloping algebra (quantum group) $U_q(\g)$ is the associative unital $\mc(q)$-algebra with generators $E_i, F_i, K_i, K_i^{-1}$ for $i\in I$ and relations: for $i,j\in I$,
$$K_iK_j=K_jK_i,\ \ K_iK_i^{-1}=K_i^{-1}K_i=1,$$
$$K_iE_jK_i^{-1}=q_i^{c_{ij}}E_j,\ \ K_iF_jK_i^{-1}=q_i^{c_{ij}}F_j,\ \ [E_i,F_j]=\delta_{ij}\frac{K_i-K_i^{-1}}{q_i-q_i^{-1}};$$
and for $i\neq j\in I$,
$$\sum_{r=0}^{1-c_{ij}}\bino{1-c_{ij}}{r}_{q_i}E_i^{1-c_{ij}-r}E_jE_i^r=0,\ \ \sum_{r=0}^{1-c_{ij}}\bino{1-c_{ij}}{r}_{q_i}F_i^{1-c_{ij}-r}F_jF_i^r=0.$$
\end{definition}

For $\lambda\in\pp_+$, we let $V(\lambda)$ denote the finite dimensional irreducible representation of $U_q(\g)$ of highest weight $\lambda$ and type $1$.
\par
Let $U(\g)$ be the enveloping algebra associated to $\g$ with generators $e_i, f_i, h_i$ for $i\in I$.
\par
It should be remarked that $U_q(\g)$ has a $\mathbb{Z}[q,q^{-1}]$-form which is called an integral form (for example, see Chapter 9 in \cite{CP} for details). This integral form allows us to specialize $U_q(\g)$ to any non-zero complex number $z$. We let $\lim_{q\ra z}U_q(\g)$ denote the specialized Hopf algebra. It is well known that $\lim_{q\ra 1}U_q(\g)$ is isomorphic to $U(\g)$.
\par
Moreover, finite dimensional representations of $U_q(\g)$ can be specialized: when $q$ tends to $1$, the representation $V(\lambda)$ is specialized to $V_1(\lambda)$.

\section{Braid group actions}\label{s3}

We briefly recall the braid group actions on the finite dimensional $U_q(\g)$-modules. For $i,j\in I$, when the product $c_{ij}c_{ji}$ equals to $0,1,2,3,4$, repectively, we denote $m_{ij}=2,3,4,6,\infty$.

\begin{definition}
The Artin braid group $\fb_\g$ associated to the Weyl group $W$ of $\g$ is a group generated by $\s_1,\cdots,\s_l$ and relations
$$\s_i\s_j\cdots\s_i\s_j=\s_j\s_i\cdots\s_j\s_i,$$
where lengths of words in both sides are $m_{ij}$.
\end{definition}

For example, if the Lie algebra $\g$ is of type $A_l$, then $m_{ij}=3$ if $|i-j|=1$, otherwise $m_{ij}=2$. In this case the Artin group $\fb_\g$ is the usual braid group $\fb_{l+1}$.
\par
We let $\Pi=\s_1\cdots\s_l$ be a product of generators in $\fb_\g$ and call it a Coxeter element. For an element $w$ in the Weyl group $W$ with reduced expression $w=s_{i_1}\cdots s_{i_t}$, we let $T(w)=\s_{i_1}\cdots \s_{i_t}$ be the element in $\mathfrak{B}_\g$. It is well-known that $T(w)$ is independent of the reduced expression. Let $w_0$ be the longest element in $W$. We call $\Delta=T(w_0)$ the Garside element in $\mathfrak{B}_\g$. The following proposition explains some properties concerning the Coxeter element $\Pi$.

\begin{proposition}[\cite{BS}, Lemma 5.8 and Satz 7.1]\label{Prop:Artin}
$ $\par
\begin{enumerate}
\item Let $\Delta$ be the Garside element in $\fb_\g$. Then $\Pi^h=\Delta^2$.
\item If $\g$ is not isomorphic to $\mathfrak{sl}_2$, the centre $Z(\fb_\g)$ of $\fb_\g$ is generated by $\Delta^2$. If $\g\cong \mathfrak{sl}_2$, the centre $Z(\fb_\g)$ is generated by $\s_1=\Delta$.
\end{enumerate}
\end{proposition}

We then recall different actions of $\fb_\g$ on the finite dimensional $U_q(\g)$-modules, following \cite{CP} and \cite{Saito}.

\begin{enumerate}
\item \textit{Saito's action:}
For $i\in I$, we let $U_q(\g)_i$ denote the subalgebra of $U_q(\g)$ generated by $E_i$, $F_i$ and $K_i^{\pm 1}$ which is isomorphic to $U_{q_i}(\mathfrak{sl}_2)$ and define an endomorphism $S_i\in \textrm{End}(V(n))$ by
\begin{equation}\label{Eq:si}
S_i=\textrm{exp}_{q_i^{-1}}(q_i^{-1}E_iK_i^{-1})\textrm{exp}_{q_i^{-1}}(-F_i)\textrm{exp}_{q_i^{-1}}(q_iE_iK_i)q_i^{H_i(H_i+1)/2},
\end{equation}
where $q_i^{H_i(H_i+1)/2}$ sends $v\in V(n)$ to $q_i^{m(m+1)/2}$ if $K_i\cdot v=q_i^m v$. This operator $S_i\in \textrm{End}(V(n))$ is well-defined as both $E_i$ and $F_i$ act nilpotently on $V(n)$.
\par
If a basis $\{v_0,\cdots,v_n\}$ of $V(n)$ is chosen in such a way that 
\begin{equation}\label{Eq:basis}
E_i\cdot v_0=0,\ \ F_i^{(k)}\cdot v_0=v_k,\ \ K_i\cdot v_0=q_i^nv_0,
\end{equation}
then the action of $S_i$ on $V(n)$ is given by \cite{Saito}:
\begin{equation}\label{Eq:Siaction}
S_i\cdot v_k=(-1)^{n-k}q_i^{(n-k)(k+1)}v_{n-k}.
\end{equation}

In the general case, for a finite dimensional $U_q(\g)$-module $M$, since it is a direct sum of irreducible $U_q(\g)_i$-modules, $S_i\in \textrm{End}(M)$ is well-defined. As these $S_i$ are invertible, we could consider the group generated by $\{S_i\mid i\in I\}$ in $\textrm{End}(M)$. It is proved by Saito (\emph{loc.cit.}) that the assignment $\s_i\mapsto S_i$ extends to a group homomorphism between $\fb_\g$ and the subgroup of $\Aut(M)$ generated by $\{S_i\mid i\in I\}$.
\item \textit{Lusztig's automorphism:}
There is another action of $\fb_\g$ on $U_q(\g)$ constructed by Lusztig. We let $T_i:=T_{i,1}''$ as in \cite{Lusztig}, Section 37.1.3, which satisfy the relations in the Artin braid group associated to the Weyl group of $\g$.
\end{enumerate}

\begin{proposition}[\cite{Saito}]\label{Prop:Sw}
Let $M$ be a finite dimensional $U_q(\g)$-module. Then for any $x\in U_q(\g)$, $T_i(x)=S_ixS_i^{-1}\in\End(M)$.
\end{proposition}

\section{Action of central element}\label{s5}

The main goal of this section is to compute the action of $Z(\mathfrak{B}_\g)$ on $V(\lambda)$.

\subsection{Action on extremal vectors}
The Artin braid group $\mathfrak{B}_\g$ acts on $V(\lambda)$. By Proposition \ref{Prop:Artin}, let $\theta=\Delta^2=\Pi^h$ be the generator of the centre $Z(\mathfrak{B}_\g)$; we compute the action of $\theta$ on the highest weight vector $v_\lambda$ in this subsection.
\par
Let $w_0=s_{i_1}\cdots s_{i_N}$ be a fixed reduced expression of $w_0$.

\begin{lemma}
The generator $\theta$ of $Z(\mathfrak{B}_\g)$ admits the following expression:
$$\theta=\s_{i_1}\cdots \s_{i_N}\s_{i_N}\cdots \s_{i_1}.$$
\end{lemma}

\begin{proof}
Since $w_0^2=1$, its inverse $w_0^{-1}=s_{i_N}\cdots s_{i_1}$. Lifting them to $\fb_\g$ gives $\Delta=\s_{i_N}\cdots \s_{i_1}=\s_{i_1}\cdots \s_{i_N}$, hence $\theta=\Delta^2$ has the desired form.
\end{proof}

Return to our situation, when acting on $V(\lambda)$, the central element $\theta\in Z(\fb_\g)$ has the following expression
$$\theta = S_{i_1}\cdots S_{i_{N-1}}S_{i_{N}}S_{i_N}S_{i_{N-1}}\cdots S_{i_{1}}.$$
\begin{proposition}\label{Prop:theta}
We have $\theta\cdot v_\lambda=(-1)^{(\lambda,2\rho^\vee)} q^{(\lambda,2\rho)}v_\lambda.$
\end{proposition}

\begin{proof}
Since $\theta\cdot v_\lambda=S_{i_1}\cdots S_{i_{N-1}}S_{i_{N}}S_{i_N}S_{i_{N-1}}\cdots S_{i_{1}}v_\lambda$, we apply repeatly the formula (\ref{Eq:si}) to compute the coefficient before $v_\lambda$.
\par
Applying $S_{i_1}$ on $v_\lambda$ gives 
$$S_{i_1}\cdot v_\lambda=(-1)^{(\lambda,\alpha_{i_1}^\vee )}q^{(\lambda,\alpha_{i_1})}v_{s_{i_1}(\lambda)}.$$
The action of $S_{i_2}$ on $v_{s_{i_1}(\lambda)}$ gives
$$S_{i_2}\cdot v_{s_{i_1}(\lambda)}=(-1)^{(s_{i_1}(\lambda),\alpha_{i_2}^\vee)} q^{(s_{i_1}(\lambda),\alpha_{i_2})}v_{s_{i_2}s_{i_1}(\lambda)}=(-1)^{(\lambda,s_{i_1}(\alpha_{i_2}^\vee))} q^{(\lambda,s_{i_1}(\alpha_{i_2}))}v_{s_{i_2}s_{i_1}(\lambda)}.$$
Repeat this procedure and notice that by (\ref{Eq:si}), for $k=1,\cdots,N$,
$$\text{the action of}\ \ S_{i_k}\ \ \text{fixes}\ \ v_{s_{i_{k+1}}\cdots s_{i_N} s_{i_N}\cdots s_{i_1}(\lambda)},$$
we obtain that $\theta\cdot v_\lambda=\mu v_\lambda$ where $\mu=(-1)^{\mu_1}q^{\mu_2}$,
$$\mu_1=\sum_{r=1}^N (\lambda,s_{i_1}\cdots s_{i_{r-1}}(\alpha_{i_r}^\vee))\ \ \text{and}\ \  \mu_2=\sum_{r=1}^N (\lambda,s_{i_1}\cdots s_{i_{r-1}}(\alpha_{i_r})).$$
When $r$ runs from $1$ to $N$, $s_{i_1}\cdots s_{i_{r-1}}(\alpha_{i_r})$ runs through the positive roots, from which the proposition holds.
\end{proof}

As a central element, $\theta$ acts by the same constant on each $\fb_\g$-orbit in $V(\lambda)$.

\subsection{Central automorphism action}
Let $U_q^{\geq 0}(\g)$, (resp. $U_q^{\leq 0}(\g); U_q^{<0}(\g)$) denote the sub-algebra of $U_q(\g)$ generated by $E_i,K_i^{\pm 1}$ (resp. $F_i, K_i^{\pm 1}$; $F_i$). We compute the action of $T_{w_0}^2$ on PBW root vectors of $U_q^{\leq 0}(\g)$ in this subsection. It is known that $T_{w_0}$ permutes $U_q^{\geq 0}(\g)$ and $U_q^{\leq 0}(\g)$, so $T_{w_0}^2$ is an automorphism of $U_q^{\leq 0}(\g)$ and of $U_q^{\geq 0}(\g)$.
\par
For $i\in I$, we let $i'$ denote the index satisfying $w_0(\alpha_i)=\alpha_{i'}$.

\begin{lemma}
For $i\in I$, the following identities hold:
$$T_{w_0}^2(E_i)=q^{(\alpha_i,\alpha_i)}K_i^{-2}E_i,\ \ T_{w_0}^2(F_i)=q^{(\alpha_i,\alpha_i)}K_i^2F_i,\ \ T_{w_0}^2(K_i)=K_i.$$
\end{lemma}

\begin{proof}[Proof]
A similar computation as in \cite{Lusztig2}, Section 5.7 gives
$$T_{w_0}(E_i)=-F_{i'}K_{i'},\ \ T_{w_0}(F_i)=-K_{i'}^{-1}E_{i'},\ \ T_{w_0}(K_i)=K_{i'}^{-1}.$$
Then the lemma is clear as $w_0(\alpha_{i'})=\alpha_i$.
\end{proof}

For $k=1,\cdots,N$, we denote $\beta_k:=s_{i_1}\cdots s_{i_{k-1}}(\alpha_{i_k})$, then $\Delta_+=\{\beta_1,\cdots,\beta_N\}$. We turn to consider the action of $T_{w_0}^2$ on a root vector $F_{\beta_k}=T_{i_1}\cdots T_{i_{k-1}}(F_{i_k})$. As $T_{w_0}^2\in Z(\fb_\g)$ (here $\fb_\g$ is the Artin braid group generated by $\{T_i\mid i\in I\}$),
\begin{eqnarray*}
T_{w_0}^2T_{i_1}\cdots T_{i_{k-1}}(F_{i_k}) &=& T_{i_1}\cdots T_{i_{k-1}} T_{w_0}^2(F_{i_k})\\
&=& T_{i_1}\cdots T_{i_{k-1}}\left(q^{(\alpha_{i_k},\alpha_{i_k})}K_{i_k}^2 F_{i_k}\right) \\
&=& q^{(\beta_k,\beta_k)}K_{\beta_k}^2 F_{\beta_k},
\end{eqnarray*} 
where $q^{(\alpha_{i_k},\alpha_{i_k})}=q^{(\beta_k,\beta_k)}$ since the bilinear form is $W$-invariant.
\par
In general, we have
\begin{eqnarray*}
T_{w_0}^2(F_{\beta_{j_1}}\cdots F_{\beta_{j_t}}) &=& q^{\sum_{k=1}^t (\beta_{j_k},\beta_{j_k})}K_{\beta_{j_1}}^2F_{\beta_{j_1}}\cdots K_{\beta_{j_t}}^2F_{\beta_{j_t}}\\
&=&
q^{\sum_{k=1}^t (\beta_{j_k},\beta_{j_k})+\sum_{i<k}2(\beta_{j_i},\beta_{j_k})}K_{\beta_{j_1}}^2\cdots K_{\beta_{j_t}}^2F_{\beta_{j_1}}\cdots F_{\beta_{j_t}}\\
&=& q^{(\beta,\beta)}K_\beta^2 F_{\beta_{j_1}}\cdots F_{\beta_{j_t}},
\end{eqnarray*}
where $\beta=\beta_{j_1}+\cdots+\beta_{j_t}$. These computations give the following

\begin{proposition}\label{Prop:w0}
Let $x_\beta\in U_q^{<0}(\g)_{-\beta}$. Then 
$$T_{w_0}^2(x_\beta)=q^{(\beta,\beta)}K_\beta^2 x_\beta.$$
\end{proposition}

\subsection{Central element acting on Weyl group orbits}

\begin{proposition}\label{Prop:main}
For any non-zero vector $v\in V(\lambda)_0$,
$$\theta\cdot v=(-1)^{(\lambda,2\rho^\vee)} q^{(\lambda,\lambda+2\rho)}v.$$
\end{proposition}

\begin{proof}[Proof]
It should be remarked that if $\lambda$ is not in the root lattice $\mathcal{Q}$, there will be no non-zero vector of weight $0$. So if $v\in V(\lambda)_0$ is a non-zero vector, $\lambda\in \mathcal{Q}_+$ and there exists $x\in U_q^{<0}(\g)_{-\lambda}$ such that $v=x\cdot v_\lambda$.
\par
By Proposition \ref{Prop:Sw}, \ref{Prop:theta} and \ref{Prop:w0}, we have the following computation:
\begin{eqnarray*}
\theta\cdot v &=& S_{i_1}\cdots S_{i_N}S_{i_N}\cdots S_{i_1}x\cdot v_\lambda\\
&=& T_{i_1}\cdots T_{i_N}T_{i_N}\cdots T_{i_1}(x)\theta\cdot v_\lambda\\
&=& T_{w_0}^2(x)\theta\cdot v_\lambda\\
&=& q^{(\lambda,\lambda)}K_\lambda^2 x\theta\cdot v_\lambda\\
&=& (-1)^{(\lambda,2\rho^\vee)} q^{(\lambda,\lambda+2\rho)}v.
\end{eqnarray*}
\end{proof}

Moreover, this method can be applied to compute the action of the central element on each $\fb_\g$-orbit. For example, if $\g=\mathfrak{sl}_3$ and $\lambda=\alpha_1+\alpha_2$, then $V(\lambda)$ is the adjoint representation of dimension $8$. The central element $\theta$ acts as $q^4$ on the $1$-dimensonal weight spaces and $q^6$ on the zero-weight space.

\subsection{Trace of Coxeter element}
We compute the trace of the Coxeter element acting on $V(\lambda)$.
\par
Let $\textrm{wt}(V(\lambda))$ denote the set of weights appearing in $V(\lambda)$. As
$$S_i(V(\lambda)_\mu)\subset V(\lambda)_{\mu-(\mu,\alpha_i^\vee)\alpha_i},$$
the action of the Artin braid group $\fb_\g$ on $\textrm{wt}(V(\lambda))$ coincides with that of the Weyl group $W$. The action of $\Pi$ is the same as the Coxeter element $c=s_1\cdots s_l\in W$ which has no fixed point on $\textrm{wt}(V(\lambda))\backslash\{0\}$.
\par
A standard proof of the statement above can be found in \cite{Bour4-6}, Chapitre V, n$^\textrm{o}$ 6.2.
\par
As an immediate consequence of this observation, we have
$$\Tr (\Pi,V(\lambda))=\Tr (\Pi,V(\lambda)_0).$$
Moreover, generators of $\fb_\g$ preserve the zero-weight space $V(\lambda)_0$, so we can also look $\fb_\g$ as a subgroup of $\textrm{Aut}(V(\lambda)_0)$.
\par
Notice that $\theta=\Pi^h$, so by Proposition \ref{Prop:main}, $\Pi^h$ acts as the scalar $(-1)^{(\lambda,2\rho^\vee)}q^{(\lambda,\lambda+2\rho)}$ on $V(\lambda)_0$. If we let $\Lambda$ denote the set of roots of the equation $x^h=(-1)^{(\lambda,2\rho^\vee)}$ in $\mathbb{C}$, then the eigenvalues of $\Pi$ belong to the set 
$$\{y.q^{\frac{(\lambda,\lambda+2\rho)}{h}}|\ y\in\Lambda\}$$
and the trace $\Tr(\Pi,V(\lambda)_0)$ is given by $\delta q^{\frac{(\lambda,\lambda+2\rho)}{h}}$ for some $\delta\in\mc$. As a summary, we have proved that

\begin{proposition}\label{Prop:delta}
There exists a constant $\delta\in\mc$ which independent with $q$ such that 
$$\Tr (\Pi,V(\lambda))=\delta q^{\frac{(\lambda,\lambda+2\rho)}{h}}.$$
\end{proposition}

\begin{remark}
To make the notation $q^{1/h}$ valid, we should enlarge the base field to $\mc(q^{1/h})$. Since it is nothing but a notational change, we ignore it in the following argument.
\end{remark}

\section{Main theorem}\label{s6}

Let $\wt{W}$ be the Tits extension of the Weyl group $W$ obtained by taking $t_i=1$ in the formula (2.13) of \cite{ML}. In the classical limit $q\ra 1$, the action of $\fb_\g$ on $V(\lambda)$ specializes to an action of $\wt{W}$ on $V_1(\lambda)$ by sending $S_i$ in (\ref{Eq:si}) to $\wt{s}_i=\exp(e_i)\exp(-f_i)\exp(e_i)$.

\par

Let $V_1(\lambda)_0$ be the subspace of $V_1(\lambda)$ of weight $0$. The Weyl group $\wt{W}$ acts on $V_1(\lambda)$ and therefore on $V_1(\lambda)_0$. Let $\bold{c}=\wt{s}_1\cdots \wt{s}_l$ be a Coxeter element in $\wt{W}$.

\begin{theorem}[\cite{Kostant}, Theorem 4.1]\label{Thm:Kostant}
The following identity holds:
$$\left(\prod_{i=1}^l\vp(x^{h\Phi(\alpha_i,\alpha_i)})\right)^{h+1}=\sum_{\lambda\in\pp_+}\Tr (\bold{c},V_1(\lambda)_0)\dim V_1(\lambda)x^{c(\lambda)}.$$
In particular, if $\g$ is simply laced (i.e. of type A,D,E), the identity above has the form
$$\vp(x)^{\dim\g}=\sum_{\lambda\in\pp_+}\Tr (\bold{c},V_1(\lambda)_0)\dim V_1(\lambda)x^{c(\lambda)}.$$
Moreover, $\Tr (\bold{c},V_1(\lambda)_0)\in\{-1,0,1\}$.
\end{theorem}
We denote $\ve(\lambda):=\Tr (\bold{c},V_1(\lambda)_0)$.
\par

We will give a compact form of the identity in Theorem \ref{Thm:Kostant} using quantum groups.

Let 
$$\mc_q[G]=\bigoplus_{\lambda\in\pp_+}\textrm{End}(V(\lambda))=\bigoplus_{\lambda\in\pp_+}V(\lambda)\ts V(\lambda)^*$$
be the quantum coordinate algebra which can be viewed as a deformation of the algebra of regular functions of a semi-simple algebraic group $G$.
\par
It is clear that there is a canonical embedding
$$\bigoplus_{\lambda\in\mathcal{P}_+}\textrm{End}(V(\lambda))\ts \textrm{End}(V(\lambda)^*)\ra \textrm{End}(\mc_q[G]).$$
The following is the main theorem of the paper.

\begin{theorem}\label{Thm:main}
Let $\Pi$ be a Coxeter element in the Artin braid group $\fb_\g$ and $V(\lambda)_0$ be the zero-weight space in $V(\lambda)$ for $\lambda\in\pp_+$.
\begin{enumerate}
\item We have
$$\Tr (\Pi,V(\lambda))=\Tr (\Pi,V(\lambda)_0)=\ve(\lambda) q^{r_\g c(\lambda)}.$$
\item The following identity holds
$$\Tr (\Pi\ts\id,\mc_q[G])=\left(\prod_{i=1}^l\vp(q^{(\alpha_i,\alpha_i)})\right)^{h+1},$$
where we look $\Pi\ts\id$ as in $\End(\mc_q[G])$ through the embedding above.
\item In particular, if $\g$ is simply laced, i.e., of type $A,D,E$, then
$$\Tr (\Pi\ts\id,\mc_q[G])=\vp(q^2)^{\dim\g}.$$
\end{enumerate}
\end{theorem}

\begin{proof}
\begin{enumerate}
\item By Proposition \ref{Prop:delta}, there exists a constant $\delta\in\mc$ such that 
$$\Tr (\Pi,V(\lambda))=\delta q^{\frac{(\lambda,\lambda+2\rho)}{h}}=\delta q^{r_\g c(\lambda)}.$$
To determine it, we consider the specialization of $\Pi$ and $V(\lambda)$. When $q$ is specialized to $1$, the left hand side has limit $\Tr (\bold{c},V_1(\lambda))$. On the other hand, as $\delta\in\mc$, the right hand side has limit $\delta$, from which $\delta=\ve(\lambda)$.
\item This point holds by the following computation using (1) and Theorem \ref{Thm:Kostant}:
\begin{eqnarray*}
\Tr (\Pi\ts\id,\mc_q[G]) &=&  \sum_{\lambda\in\pp_+}\Tr \left(\Pi\ts\id,V(\lambda)\ts V(\lambda)^*\right)\\
&=& \sum_{\lambda\in\pp_+}\textrm{dim}_{\mc(q)}V(\lambda) \Tr (\Pi,V(\lambda))\\
&=& \sum_{\lambda\in\pp_+}\ve(\lambda)\textrm{dim}_{\mc(q)}V(\lambda) q^{r_\g c(\lambda)}\\
&=& \left(\prod_{i=1}^l\vp(q^{r_\g h\Phi(\alpha_i,\alpha_i)})\right)^{h+1}\\
&=& \left(\prod_{i=1}^l\vp(q^{(\alpha_i,\alpha_i)})\right)^{h+1}.
\end{eqnarray*}
\item Notice that in the simply laced case, $(\alpha_i,\alpha_i)=2$ and $l(h+1)=\textrm{dim}\g$.
\end{enumerate}
\end{proof}

We finally examine the theorem when $\g=\mathfrak{sl}_2$. In this case, there is only one generator $S\in\fb_\g$ so the Coxeter element is $S$. We will use the basis of $V(n)$ in (\ref{Eq:basis}).
\begin{enumerate}
\item If $n$ is odd, there is no zero-weight space in $V(n)$, in this case, $\Tr (S,V(n))=0$.
\item If $n$ is even, the zero-weight space in $V(n)$ is of dimension $1$ which is generated by $v_m$, where $n=2m$. The action of $S$ on $v_m$ is given by
$$S\cdot v_m=(-1)^m q^{(n-m)(m+1)}v_m=(-1)^mq^{m(m+1)}v_m.$$
\end{enumerate}
As a conclusion,
$$\Tr (S\ts\id,\mc_q[G])=\sum_{m=0}^\infty (-1)^m (2m+1) q^{m(m+1)},$$
which coincides with $\vp(q^2)^3$, verifying the Jacobi's identity.

\end{document}